\documentclass{amsart}

\usepackage{amscd}
\usepackage{amssymb}
\usepackage{mathrsfs}
\usepackage{amsmath}
\usepackage{mathabx}
\usepackage{amsthm}
\usepackage[all,cmtip]{xy}
\usepackage{enumerate}
\usepackage{enumitem}
\usepackage{tikz}
\usepackage{mathtools}
\usepackage{color}
\usepackage{wasysym}

\numberwithin{equation}{section}

\newcommand{\thmref}[1]{Theorem~\ref{#1}}
\newcommand{\secref}[1]{Section~\ref{#1}}
\newcommand{\appref}[1]{Appendix~\ref{#1}}
\newcommand{\lemref}[1]{Lemma~\ref{#1}}
\newcommand{\propref}[1]{Proposition~\ref{#1}}

\newcommand{\eqnref}[1]{(\ref{#1})}

\newtheorem{theorem}{Theorem}[section]
\newtheorem{lemma}[theorem]{Lemma}
\newtheorem{proposition}[theorem]{Proposition}
\newtheorem{corollary}[theorem]{Corollary}

\theoremstyle{definition}

\newtheorem{definition}[theorem]{Definition}

\theoremstyle{remark}
\newtheorem{remark}[theorem]{Remark}

\newcommand{\qbinom}[2]{\genfrac{[}{]}{0pt}{}{#1}{#2}}
\newtheorem{conjecture}[theorem]{Conjecture}
\newcommand{\C}{\mathbb{C}}
\newcommand{\F}{\mathbb{F}}
\newcommand{\Z}{\mathbb{Z}}

\newcommand{\N}{\mathbb{N}}

\newcommand{\fb}{\mathfrak{b}}
\newcommand{\fg}{\mathfrak{g}}
\newcommand{\fh}{\mathfrak{h}}

\newcommand{\gl}{\mathfrak{gl}}
\newcommand{\fsl}{\mathfrak{sl}}

\newcommand{\CC}{\mathcal{C}}

\newcommand{\U}{{\rm U}}

\newcommand{\End}{{\rm End}}

\begin{document}

\title[Centre of the quantum group]{Explicit generators of the centre of \\ the quantum group}
 \author[Yanmin Dai]{Yanmin Dai}
\address{School of Mathematical Sciences, University of Science and Technology of China, Heifei, 230026, China}
\email{bt2@mail.ustc.edu.cn}

\begin{abstract}
A finite generating set of the centre of any quantum group is obtained, where the generators are given by an explicit formulae. For the slightly generalised version of the quantum group which we work with, we show that this set of generators is algebraically independent, thus the centre is isomorphic to a polynomial algebra. 
\end{abstract}
 
\maketitle
  
\tableofcontents
  
\section{Introduction}
  
We construct explicit generators of the centre of any quantum group in this paper.  This involves two separate problems, namely, 
the construction, in terms of explicit formulae, of a finite set of central elements,  and 
the proof that they generate the centre of the quantum group.

Note that even in the classical case of a semi-simple Lie algebra $\fg$, where the algebraic structure of the centre of the universal enveloping algebra $\U(\fg)$ is well understood, it is an important and highly non-trivial problem to construct explicit generators of the centre of $\U(\fg)$, that is, to construct the (high order) Casimir operators.

The explicit generators of the centre of the quantum group to be constructed in this paper are quantum analogues of (higher order) Casimir operators of $\U(\fg)$ arising from ``characteristic identities" \cite{BG}. The operators, both classical and quantum, are particularly useful for explicit computations in representation theory, e.g., for computing Wigner coefficients and developing Racah-Wigner calculus (see, e.g., \cite{G1, WIG}), and for solving Hamiltonian systems in atomic and molecular physics.

We choose to work with a slightly generalised Jimbo version of the quantum group (see \cite{DK,DP,JL} and Definition \ref{def} in particular). Recall that the standard quantum group has among its generators $K_{\alpha_i}^{\pm 1}$ associated with the simple roots $\alpha_i$, thus contains all $K_\beta^{\pm 1}$ for $\beta$ in the root lattice. The generalised quantum group $\U_q(\fg)$ includes all the $K_\mu^{\pm 1}$ for $\mu$ in the weight lattice; see Remark \ref{rem} for more details.

The main results of the paper are summarised in Theorem \ref{mainresult} and Corollary \ref{cor:poly}.
Let  $rk(\fg)$ be the rank of $\fg$. In Theorem \ref{mainresult}, we give a set of $rk(\fg)$ central elements of $\U_q(\fg)$ in terms of the explicit formulae \eqref{eq:key-formula}, and show that they generate the centre of $\U_q(\fg)$.  In Corollary \ref{cor:poly}, we summarise that the set of generators are algebraically independent, thus the centre of the generalised quantum group $\U_q(\fg)$ is isomorphic to a polynomial algebra of $rk(\fg)$ variables. 

We now briefly describe our approach to the results. 

Recall that an infinite family of central elements of $\U_q(\fg)$ were constructed from any given finite dimensional representation $(V, \zeta)$ (where $V$ is a $\U_q(\fg)$-module and $\zeta$ is the corresponding representation)  of the quantum group in \cite{zhang2, zhang1}.  The construction makes use of the universal $L$-operators $L_V, L_V^T\in \End(V)\otimes \U_q(\fg)$ (see \eqref{eq: LvLvT}), the existence of which in the current version of $\U_q(\fg)$ is explained in \secref{sec:D}. The operator $\Gamma_V= L_V^T L_V$ (see \eqnref{eq: defgamma}) commutes with $(\zeta\otimes{\rm id})\Delta(x)$ for all $x\in \U_q(\fg)$, where $\Delta$ is the comultiplication, thus by \thmref{thm:BGZ} (see also \cite[Proposition 1]{zhang2}), the $q$-trace $\CC_V^{(m)}$ of $\Gamma_V^m$ is a central element  of  $\U_q(\fg)$ for each positive integer $m$.  The generating set of the centre of $\U_q(\fg)$ given in Theorem \ref{mainresult} consists of the elements $\CC_V^{(1)}$ for $V$ being the simple $\U_q(\fg)$-modules with fundamental highest weights. 

We prove that the set of central elements given in Theorem \ref{mainresult} generates the centre by 
using a quantum Harish-Chandra isomorphism (see \thmref{thm:HCiso} and \cite[\S 18.3]{DP}) for the generalised quantum group.  In particular, we prove that the images of those central elements  generate the image of the centre of the quantum group under the quantum Harish-Chandra isomorphism. We also show that the images are algebraically independent, thus generate a polynomial algebra, leading to Corollary \ref{cor:poly}.

This quantum Harish-Chandra isomorphism is an analogue of a similar result \cite{Jantzen,T} for the standard quantum groups. It was presented in \cite[\S 18.3]{DP}, where a geometric proof was given following familiar arguments in the theory of semi-simple Lie algebras. (There was also an indication of the result in the introduction section of \cite{Lg}.)  Due to its importance to us, we give an elementary algebraic proof of the result in \appref{app: proof} by adapting \cite[Chapter 6]{Jantzen} to our context.  The algebraic proof works in much the same way for both the standard and generalised quantum groups, however it is not clear to us how the geometric method of \cite[\S 18.3]{DP} would work in the case of the standard quantum groups. 

The authors of \cite{zhang1, zhang2} expected that some finite subset of the central elements $\CC_V^{(n)}$ for all finite dimensional simple modules $V$ and all $n=1, 2, \dots$ generates the centre of $\U_q(\fg)$, but this was not proved before except for $\fg=\gl_\ell$ \cite{Junbo}.  In Conjecture \ref{conj}, we suggest another subset of the elements $\CC_V^{(n)}$ as a generating set of the centre of $\U_q(\fg)$. For $\fg=\gl_\ell$, the conjecture is implied by the results of \cite{Junbo}. 

In a future publication,  we will give a similar treatment of the centres of the standard quantum groups and quantum supergroups \cite{BGZ}. 
We point out that complete sets of generators were constructed for the centres of  the standard quantum groups $\U_q(\fsl_3)$ and $\U_q(\fsl_4)$ \cite{L3, L4}. 

The paper is organized as follows. In \secref{sec: QuanGrp}, we introduce the basics of the generalised Jimbo quantum group, which we still denote by $\U_q(\fg)$, and give the new Harish-Chandra isomorphism for generalised quantum group $\U_q(\fg)$. In \secref{sec:D}, we give an analogue of universal $R$-matrix in finite dimensional representations of $\U_q(\fg)$,  and use them to construct infinite families of central elements of the generalised quantum group.
Finally in \secref{sec: main-result}, we extract from these central elements a generating set of the centre of the generalised quantum group. This is  explained in Theorem \ref{mainresult}. The appendix contains the algebraic proof of the quantum Harish-Chandra isomorphism for $\U_q(\fg)$ (see Theorem \ref{thm:HCiso}).

\section{Basics on quantum groups}

\subsection{A generalisation of the Jimbo quantum group}\label{sec: QuanGrp}

Let $\fg$ be a finite dimensional simple Lie algebra over the field of complex numbers $\C$.  Choose Borel and Cartan subalgebras $\mathfrak{h}\subseteq \fb\subseteq \mathfrak{g}$, and fix a basis $H_1,\dots,H_n$ for the Cartan subalgebra $\fh$, where $n$ is the rank of $\fg$. Denote by $\Phi$ the root system of $\mathfrak{g}$ relative to this choice, and let $\Phi^+$ be the set of positive roots. Let $\Pi=\{\alpha_1,\alpha_2,\cdots,\alpha_n\}$ be the simple root system of $\Phi$. Let $(\ ,\ )$ be the non-degenerate bilinear form on the dual space $\mathfrak{h}^*$ normalised so that the square of the length of short roots is $2$. Denote by $W$ the Weyl group of $\mathfrak{g}$. The Cartan matrix $A=(a_{ij})$ is the $n\times n$ matrix with $a_{ij}=\frac{2(\alpha_i,\alpha_j)}{(\alpha_i,\alpha_i)}$. Moreover, we set $d_{i}=(\alpha_i,\alpha_i)/2$  and let $\alpha_i^{\vee}=d_i^{-1}\alpha_i$ be the simple coroot for $1\leq i\leq n$.

Denote by $\varpi_i, 1\leq i\leq n$ the fundamental weights of $\mathfrak{g}$ such that $\frac{2(\varpi_i, \alpha_j)}{(\alpha_j, \alpha_j)}=\delta_{ij}$ for all $i, j$. We write
\[
 P=\bigoplus_{i=1}^{n} \mathbb{Z}\varpi_i, \quad Q= \bigoplus_{i=1}^n \mathbb{Z}\alpha_i
\]
for the weight and root lattices of $\mathfrak{g}$, respectively.  Clearly $(\lambda, \alpha)\in \Z$ for all $\lambda\in P$ and $\alpha\in\Phi$ for the given  normalisation of the bilinear form.
Let $P^+=\bigoplus_{i=1}^n \N \varpi_i$ be the set of dominant weights, and $Q^{+}=\bigoplus_{i=1}^n \N\alpha_i $ the set of non-negative integer combinations of simple roots.

We denote by $\ell_\fg$ the minimal positive integer such that $\ell_\fg (\lambda, \mu)\in\Z$ for all $\lambda, \mu\in P$.
Let $q^{\frac{1}{\ell_\fg}}$ be an indeterminate, and denote by $\C(q^{\frac{1}{\ell_\fg}})$ the field of rational functions. We will always write
\[
\F=\C(q^{\frac{1}{\ell_\fg}})\quad \text{and}\quad  q=\left(q^{\frac{1}{\ell_\fg}}\right)^{\ell_\fg}.
\]
For any
$\lambda, \mu\in P$, the expression $q^{(\lambda, \mu)}$ will mean $\left(q^{\frac{1}{\ell_\fg}}\right)^{\ell_\fg(\lambda, \mu)}$.

\begin{definition}\label{def}
  The generalised Jimbo quantum group associated with $\fg$ is the unital associative algebra over $\F$ generated by $E_{i},F_{i}  (i=1,2,\cdots,n)$ and  $K_{\lambda}  (\lambda\in P)$ subject to the following relations:
% \begin{equation}\label{eq: Qrel}
  \begin{align}
     K_{0}=1,\quad K_{\lambda}&K_{\mu}=K_{\lambda+\mu}, \label{eq: Qrel1}\\
     K_{\lambda}E_{j}K_{\lambda}^{-1}=&q^{(\lambda,\alpha_j)}E_{j}, \label{eq: Qrel2}\\
    K_{\lambda}F_{j}K_{\lambda}^{-1}=&q^{-(\lambda,\alpha_j)}F_{j}, \label{eq: Qrel3}\\
    E_{i}F_{j}-F_{j}E_{i}=&\delta_{ij}\frac{K_{i}-K_{i}^{-1}}{q_i-q_i^{-1}}, \label{eq: Qrel4}\\
    \sum\limits_{s=0}^{1-a_{ij}}(-1)^s \qbinom{1-a_{ij}}{s}_{q_i} &
E_{i}^{1-a_{ij}-s}E_{j}E_{i}^{s}=0,\ i\neq j, \label{eq: Qrel5}\\
    \sum\limits_{s=0}^{1-a_{ij}}(-1)^s \qbinom{1-a_{ij}}{s}_{q_i}&
F_{i}^{1-a_{ij}-s}F_{j}F_{i}^{s}=0,\ i\neq j,\label{eq: Qrel6}
 \end{align}
% \end{equation}
where $K_i=K_{\alpha_i}$, $q_i=q^{d_i}$ and
\begin{equation}\label{eq: qbinomial}
  [m]_{q_i}=\frac{q_i^{m}-q_i^{-m}}{q_i-q_i^{-1}},
  \quad[m]_{q_i}!=[1]_{q_i}[2]_{q_i}\cdots [m]_{q_i},
  \quad \qbinom{m}{k}_{q_i}=\frac{[m]_{q_i}!}{[m-k]_{q_i}![k]_{q_i}!},
  \end{equation}
for any $m\in\N$.
\end{definition}

\begin{remark}\label{rem}
The standard Jimbo quantum groups has generators $E_i$, $F_i$ and $k_i^{\pm 1}$, where $k_i$ correspond to our $K_{\alpha_i}$ for simple roots $\alpha_i$. The structure and representation theories of these two types of quantum groups are similar. However, for our purpose it is more convenient to work with Definition \ref{def}.
\end{remark}

\begin{remark}
By a slight abuse of notation and terminology, we shall denote the algebra in Definition \ref{def} by $\U_q(\fg)$, and
simply refer to it as the Jimbo quantum group.
\end{remark}

We will write $\U=\U_q(\fg)$.
It is well known that $\U$ is a Hopf algebra with  co-multiplication $\Delta$, co-unit $\varepsilon$ and antipode $S$:
\begin{eqnarray*}
&\Delta(K_{\lambda})=K_{\lambda}{\otimes}K_{\lambda}, \quad
\Delta(E_{i})=K_{i}{\otimes}E_{i}{+}E_{i}{\otimes}1, \quad
\Delta(F_{i})=F_{i}{\otimes}K_{i}^{-1}{+}1{\otimes}F_i,\\
&\varepsilon (K_{\lambda})=1,\quad \varepsilon (E_{i})=0,
\quad \varepsilon({F_{i}})=0,\\
&S(K_{\lambda})=K_{\lambda}^{-1},\quad S(E_{i})={-}K_{i}^{{-}1}E_{i},\quad S(F_{i})={-}F_{i}K_{i}.
\end{eqnarray*}
We will use Sweedler notation for the co-multiplication: $\Delta(x)=\sum_{(x)}x_{(1)}\otimes x_{(2)}$  for any $x\in {\rm U}$. The adjoint representation ${\rm ad}$ of ${\rm U}$ is defined as follows:
\begin{align}
\label{sec1.adad}
\textnormal{ad}(x)(y)=\sum_{(x)}x_{(1)}yS(x_{(2)}), \quad \forall x,y\in {\rm U}.
\end{align}

We denote by ${\rm U}^+$ (resp. ${\rm U}^-$) the subalgebra of ${\rm U}$ generated by all $E_{i}$ (resp. $F_{i}$), and by  ${\rm U}^0$ the subalgebra generated by $K_{\lambda}$ with $\lambda\in P$. Then the multiplication in ${\rm U}$ induces the  isomorphism ${\rm U}^-\otimes {\rm U}^0 \otimes {\rm U}^+\cong {\rm U}$.

\subsection{Representations}
\label{sec2.representations}
We follow \cite{Jantzen, Lu} to discuss the the  representation theory of ${\rm U}$.
Let $\sigma: P \to \mathbb{C}^{\times}$ be the group character such that
\begin{equation}\label{eq: gpchar}
  \begin{aligned}
 &\sigma(0)=1,\quad  \sigma(\lambda+\mu)=\sigma(\lambda)\sigma(\mu), \quad \forall \lambda, \mu\in P,\\
 & \sigma(\alpha_i)^2=1,\, i=1,\dots n.
\end{aligned}
 \end{equation}
There exists a one-dimensional ${\rm U}$-module $\F_{\sigma}$ associated with $\sigma$ defined by
\[
E_i.1=F_i.1=0, \quad  K_\mu.1=\sigma(\mu)1, \quad \forall\ 1\leq i\leq n,\ \mu\in P.
\]
We also have the following algebra automorphism $\psi_{\sigma}: {\rm U}\longrightarrow {\rm U}$
\begin{align}
\label{sec2.widehatsigma}
  \begin{array}{cccc}
  \psi_{\sigma}(E_{i})=\sigma(\alpha_i)E_{i},&\psi_{\sigma}(F_{i})=F_{i},&\psi_{\sigma}(K_{\mu})=\sigma(\mu)K_{\mu},
  \end{array}
\end{align}
for $1\leq i\leq n$ and  $\mu\in P$.

Let $V$ be a (left) module over ${\rm U}$. For any $\lambda\in P$ we define the weight space
\[ V_{\lambda, \sigma}=\{ v\in V | K_{\mu}.v=\sigma(\mu) q^{(\lambda,\mu)} v, \, \forall \mu \in P \}, \]
where $\sigma: P \to \mathbb{C}^{\times}$ is the group character as defined in \eqref{eq: gpchar}. It is well known that every finite dimensional $\U$-module is a weight module, i.e., a direct sum of its weight spaces.  We say $V$ is of type $\sigma$ if $V=V^{\sigma}=\oplus_{\lambda\in P} V_{\lambda, \sigma}$. Clearly, $V^\sigma=V^{{\bf 1}}\otimes \F_\sigma$, where ${\bf 1}$ is the trivial group character sending every $\lambda\in P$ to $1$. Therefore,  we will restrict our attention  to only type ${\bf 1}$ modules.

Of particular interest is the Verma module of ${\rm U}$. If $\lambda\in \mathfrak{h}^*$, we let $\chi_{\lambda}: {\rm U}^0 \to \F$ be the character such that $\chi_{\lambda}(K_{\mu})= q^{(\lambda,\mu)}$ for any $\mu\in P$. Then the Verma module $M(\lambda)$ associated to $\lambda$ is defined as the following induced ${\rm U}$-module
\[ M(\lambda): = {\rm U} \otimes_{{\rm U}^{\geq 0}}\F_{\lambda},   \]
where ${\rm U}^{\geq 0}$ is the subalgebra generated by  all $E_i$ and $K_{\mu}$, and $\F_{\lambda}$ denotes the one-dimensional ${\rm U}^{\geq 0}$-module with the action induced from $\chi_{\lambda}$  such that  all $E_i$ actions vanish. The vector $v_{\lambda}=1\otimes 1\in M(\lambda)_{\lambda}$ is called the highest weight vector with the property that
\[  E_i.v_{\lambda}=0, \quad K_{\mu}.v_{\lambda}=q^{(\lambda,\mu)} v_{\lambda}, \quad \forall 1\leq i\leq n, \mu\in P. \]
It turns out that every finite dimensional simple module of ${\rm U}$ is a quotient of $M(\lambda)$ by its unique maximal proper submodule $N(\lambda)$ for some  dominant weight $\lambda\in P^+$. We will write $V(\lambda)$ for the simple quotient module $M(\lambda)/N(\lambda)$.

For any $\lambda\in P^+$, denote by $\Pi(\lambda)$ the set of weights of the simple module $V(\lambda)$. Let $m_{\lambda}(\mu)$ be the multiplicity of $\mu$ in $V(\lambda)$, that is, $m_\lambda(\mu)=\dim V(\lambda)_{\mu}$. Then $\Pi(\lambda)$ is $W$-invariant, and
 $$m_{\lambda}(\mu)=m_{\lambda}(w \mu ), \quad \forall w\in W, \, \mu \in\Pi(\lambda).$$
 Furthermore, $m_{\lambda}(\lambda)=1$.

An important fact is  that finite dimensional modules separate the points of ${\rm U}$.

\begin{proposition}\cite[Chapter 5.11]{Jantzen}
  \label{sec1.injection}
 If an element $u\in {\rm U}$ satisfies $u.V=0$ for all finite dimensional ${\rm U}$-modules $V$, then $u=0$.
\end{proposition}

\subsection{The quantised Harish-Chandra isomorphism}\label{sec:HC}
To prepare for the construction of a generating set for the centre of the generalised quantum group $\U_q(\fg)$,
we first consider a quantised Harish-Chandra isomorphism for $\U_q(\fg)$ in this section. The discussion here will be brief, and we follow closely \cite[Chapter 6]{Jantzen}.

Let $Z({\rm U})$ denote the centre of ${\rm U}$.
Recall that ${\rm U}=\bigoplus_{\nu\in Q} {\rm U}_{\nu}$ is a  $Q$-graded algebra with
\begin{equation}\label{eq: grading}
   {\rm U}_{\nu}=\{u\in{\rm U}\, |\, \text{ad}(K_{\lambda})u=q^{(\lambda,\nu)}u, \, \forall\lambda\in P\}.
 \end{equation}
In particular, ${{\rm U}}_0={\rm U}^0\oplus\bigoplus\limits_{\nu>0}{\rm U}_{-\nu}^-{\rm U}^0{\rm U}_{\nu}^+$, where ${\rm U}^0$ is the subalgebra generated by $K_{\lambda}^{\pm 1}, \lambda\in P$. Clearly we have $Z({\rm U})\subseteq {\rm U}_0$, and the following projection
\begin{equation}
\label{sec3.definition of pi}
\pi: {\rm U}_0\to{\rm U}^0
\end{equation}
is an algebra homomorphism.

For any $\lambda\in P$, we define the twisting algebra automorphism
\begin{equation}\label{eq: gamma}
\gamma_{\lambda}:{\rm U}^0\to{\rm U}^0,\quad\gamma_{\lambda}(h)=\chi_{\lambda}(h)h, \quad \forall h\in{\rm U}^0.
\end{equation}
Denote $\rho=\frac{1}{2} \sum_{\alpha\in \Phi^+} \alpha$.  Composing \eqref{sec3.definition of pi} and \eqref{eq: gamma} with $\lambda=-\rho$, we obtain the  Harish-Chandra homomorphism $\gamma_{-\rho}\circ \pi: {\rm U}_0 \to  {\rm U}^0$, which in particular maps the centre $Z({\rm U})$ to $ {\rm U}^0$.

We proceed to characterise the image $\gamma_{-\rho}\circ \pi(Z({\rm U}))$, which will turn out to be the Weyl group invariants in ${\rm U}^0$. For any simple Lie algebra  $\mathfrak{g}$, the associated Weyl group $W$ (cf. \cite{Humpreys}) is generated by the simple reflections $s_{\alpha_i}, 1\leq i\leq n$ with
\[ s_{\alpha_i}\lambda= \lambda- (\lambda, \alpha_i^{\vee})\alpha_i, \quad \forall \lambda\in \mathfrak{h}^*. \]
This induces a natural Weyl group action on ${\rm U}^0$, i.e., $wK_{\lambda}=K_{w\lambda}$ for any $w\in W, \lambda\in P$. Let ${\rm U}_{ev}^0=\bigoplus_{\lambda\in P}\F K_{2\lambda}$ be the subalgebra of ${\rm U}^0$ generated by the even elements $K_{2\lambda}$ for all $\lambda\in P$, and define
\[
({\rm U}_{ev}^0)^W:=\{ h\in {\rm U}_{ev}^0 \mid wh=h,\, \forall w\in W \},
\]
which is the subalgebra of $W$-invariants in ${\rm U}^0$.

The following result is the quantised Harish-Chandra isomorphism for the generalised  Jimbo quantum group ${\rm U}$. It is an adaption to ${\rm U}$ of the quantised Harish-Chandra isomorphism for the standard Jimbo quantum groups established in \cite{Jantzen, T}.  

\begin{theorem}\label{thm:HCiso}
The twisted algebra homomorphism
 \begin{equation}\label{eq: HCmap}
  \gamma_{-\rho}\circ\pi:Z({\rm U})\rightarrow({\rm U}_{ev}^0)^W
 \end{equation}
is an isomorphism.
\end{theorem}

Note that the \thmref{thm:HCiso} has been stated and proven with 
a geometric method in \cite{DP}. However, we have no idea whether the method is applicable to the standard quantum groups. Since it is of crucial importance for us, we give an elementary algebraic proof of the result in \appref{app: proof} as a reference.

\section{Construction of central elements}\label{sec:D}
\subsection{The construction}

The method developed in \cite{zhang1,zhang2} plays a crucial role in the construction of central elements. It enables us to obtain an infinite family of central elements from any finite dimensional $\U_q(\fg)$-modules.

Given any finite dimensional $\U_q(\fg)$-module $V$, denote by $\zeta: \U_q(\fg)\longrightarrow \End(V)$ the associated $\U_q(\fg)$ representation. In this case, we also say that $(V, \zeta)$ is a representation of $\U_q(\fg)$.
We define the partial trace
\[
\begin{aligned}
&\textnormal{Tr}_1: \End(V)\otimes {\rm U}_q(\fg) \longrightarrow \U_q(\fg), \\
&\textnormal{Tr}_1(f\otimes x):=\text{Tr}(f)x, \quad \forall f\in\End(V), \ x\in {\rm U}_q(\fg),
\end{aligned}
\]
where ${\rm Tr}$ is the usual trace operator on $\End(V)$.

The following result is proved in \cite[Proposition 1]{zhang1}.
\begin{theorem}[\cite{zhang1}]\label{thm:BGZ}
Let $\Gamma_V\in \End(V)\otimes \U_q(\fg)$. If $\Gamma_V$ commutes with $\U_q(\fg)$ in the sense that
\[
\Gamma_V(\zeta\otimes{\rm id})\Delta(x) - (\zeta\otimes{\rm id})\Delta(x) \Gamma_V=0, \quad \forall x\in\U_q(\fg),
\]
then
\begin{eqnarray}\label{eq:key}
\textnormal{Tr}_1\left((\zeta(K_{2\rho})\otimes{\rm id})\Gamma_V\right) \in Z(\U_q(\fg)).
\end{eqnarray}
\end{theorem}
It immediately follows that
%\begin{remark}\label{rem:gamma-n}
\begin{corollary}\label{rem:gamma-n}
If $\Gamma_V\in \End(V)\otimes \U_q(\fg)$ commutes with $\U_q(\fg)$,  then
for all $m\in \N^+$,  the elements
${\rm Tr}_1((\zeta(K_{2\rho})\otimes {\rm id})\Gamma_V^m)$
belong to the centre of $\U_q(\fg)$.
\end{corollary}
%\end{remark}

\subsection{Constructing $\Gamma_V$}
Now the problem is to develop non-trivial elements $\Gamma$. We will do this following \cite{zhang1,zhang2}.

We need the explicit expression for the quasi-$R$-matrix $\mathfrak{R}$ of $\U_q(\fg)$.
Recall the braid group action on the quantum group. The braid group $\mathfrak{B}_{\mathfrak{g}}$ associated with the Weyl group $W$ of $\mathfrak{g}$ is generated by $n$ elements $\sigma_i$ with relations
\[ \sigma_i \sigma_j\sigma_i\sigma_j\dots =\sigma_j\sigma_i\sigma_j\sigma_i\dots, \quad i\neq j, \]
where the number of $\sigma$'s on each side is $m_{ij}$,   which is determined by the Cartan matrix. We have $m_{ij}=2,3,4,6$ for $a_{ij}a_{ji}=0,1,2,3$. The braid group $\mathfrak{B}_{\mathfrak{g}}$ acts as group of algebra automorphisms of ${\rm U}_q(\mathfrak{g})$. For the explicit algebra automorphism $\mathcal{T}_i$ corresponding to each generator $\sigma_i$, we refer to \cite{Lu, KR}. These are usually referred to as the Lusztig automorphisms.

Write $s_{i}$ for the simple reflection $s_{\alpha_i}$ in the Weyl group $W$ of $\mathfrak{g}$.  Let $w_0=s_{i_1}s_{i_2}\dots s_{i_N}$ be a reduced expression of the longest element of $W$. Then the positive roots of $\mathfrak{g}$ are given by the following successive actions on simple roots $\alpha_i$
\[ \beta_1=\alpha_{i_1}, \beta_2=s_{i_1}(\alpha_{i_2}), \dots, \beta_N= s_{i_1}\dots s_{i_{N-1}}(\alpha_{i_N}). \]
We define the root vectors of ${\rm U}_q(\mathfrak{g})$ by
\[ E_{\beta_r}= \mathcal{T}_{i_1}\mathcal{T}_{i_2}\dots \mathcal{T}_{i_{r-1}}(E_{i_r}), \quad F_{\beta_r}= \mathcal{T}_{i_1}\mathcal{T}_{i_2}\dots \mathcal{T}_{i_{r-1}}(F_{i_r})  \]
for all $1\leq r\leq N$. Note that $E_{\beta_r}\in {\rm U}^{+}$ and $F_{\beta_r}\in {\rm U}^{-}$. By \cite[4.1]{Lu}, we have
\begin{equation}\label{eq: tilR}
     \mathfrak{R}= \sum_{r_1,\dots, r_N=0}^{\infty} \prod_{j=1}^N q_{\beta_j}^{\frac{1}{2}r_j(r_j+1)} \frac{(1-q_{\beta_j}^{-2})^{r_j}}{[r_j]_{q_{\beta_j}}!} F_{\beta_j}^{r_j} \otimes E_{\beta_j}^{r_j} ,
\end{equation}
where $q_{\beta}:=q_i$ if $\beta$ and $\alpha_i$ lie in the same orbit under the action of $W$, the factors in the product \eqnref{eq: tilR} appear in the order $\beta_N,\beta_{n-1},\cdots,\beta_1.$
Now $\mathfrak{R}$ is an infinite sum, which belongs to some completion of $\U_q(\fg)\otimes\U_q(\fg)$.

The quasi $R$-matrix $\mathfrak{R}$ has many remarkable properties. In particular, the property discussed below will be important for us. One can easily verify that there is an algebra automorphism $\Psi$ of $\U_q(\fg)\otimes\U_q(\fg)$ defined by
\begin{align}\label{eq: Psi1}
\Psi(K_{\lambda}\otimes 1)=K_{\lambda}\otimes1,\quad&\Psi(E_i\otimes 1)=E_i\otimes K_i^{-1},\quad\Psi(F_i\otimes 1)=F_i\otimes K_i,\\ \label{eq: Psi2}
\Psi(1\otimes K_{\lambda})=1\otimes K_{\lambda},\quad&\Psi(1\otimes E_i)=K_i^{-1}\otimes E_i ,\quad\Psi(1\otimes F_i)=K_i\otimes F_i.
\end{align}
Then it is well-known that $\mathfrak{R}$ satisfies the following relation (see, e.g., \cite[\S 4.3]{T2}).

\begin{proposition}
Let $\mathfrak{R}^T= T(\mathfrak{R})$, where $T$ is the flip $T(x\otimes y)=y\otimes x$ for any $x,y \in {\rm U}_q(\fg)$. We have 
\begin{align}\label{eq: quaiRcomm}
  \mathfrak{R}\Delta(x)=\Psi(\Delta'(x))\mathfrak{R}, \quad \mathfrak{R}^T\Delta'(x)=\Psi(\Delta(x))\mathfrak{R}^T,\quad\forall x\in\U_q(\fg).
  \end{align}
\end{proposition}

Now we turn to the construction of $\Gamma_V$.
Let $(V, \zeta)$ be a finite dimensional representation of $\U_q(\fg)$.
Define
\begin{align}\label{eq: RV}
 \mathcal{R}_V:=(\zeta\otimes {\rm id})(\mathfrak{R}), \quad
\mathcal{R}_{V}^T:=(\zeta\otimes {\rm id})\mathfrak{R}^T.
\end{align}
Then both $\mathcal{R}_{V}$ and $\mathcal{R}_{V}^T$ belong to $\End(V)\otimes \U_q(\fg)$, which are well defined.

Denote by $\Pi(V)$ the set of weights of the $\U_q(\fg)$-module $V$. For any $\eta\in\Pi(V)$, we denote by $V_\eta\subset V$ the weight space of weight $\eta$.
Let $P^V_{\eta}:V\to V_{\eta}$ be the projection from $V$ to $V_{\eta}$, and define the following
element in $\End(V)\otimes\U_q(\fg)$.
\begin{align}\label{eq: KV}
\mathcal{K}_{V}:=\sum_{\eta\in\Pi(V)}P_{\eta}^V\otimes K_{\eta}.
\end{align}
Regard this as an endomorphism of $V\otimes\U_q(\fg)$ with the inverse $\mathcal{K}_{V}^{-1}=\sum_{\eta\in\Pi(V)}P_{\eta}^V\otimes K_{\eta}^{-1}$. 

Recall that the generators of $\U_q(\fg)$ are $K_{\lambda},E_i,F_i$ with $\lambda\in P,1\leq i\leq n$, and $\zeta: \U_q(\fg)\longrightarrow \End(V)$ the associated $\U_q(\fg)$ representation of $V$, then $\mathcal{K}_V$ satisfies the following relations, 
\begin{align}
\mathcal{K}_V(\zeta(K_{\lambda})\otimes K_{\mu})=(\zeta(K_{\lambda})\otimes K_{\mu})\mathcal{K}_V\label{eq: K act on V}\\
\mathcal{K}_V(\zeta(E_i)\otimes 1)=(\zeta(E_i)\otimes K_{i})\mathcal{K}_V,\ \mathcal{K}_V(1\otimes E_i)&=(\zeta(K_i)\otimes E_i)\mathcal{K}_V,\label{eq: K act on EV}\\
\mathcal{K}_V(\zeta(F_i)\otimes 1)=(\zeta(F_i)\otimes K_{i}^{-1})\mathcal{K}_V,\ \mathcal{K}_V(1\otimes F_i)&=(\zeta(K_i^{-1})\otimes F_i)\mathcal{K}_V.\label{eq: K act on FV}
\end{align}
Recall that $\Psi$ is the algebra isomorphism defined in \eqnref{eq: Psi1}, \eqnref{eq: Psi2}, and so is $(\zeta\otimes {\rm id})\Psi$,
One can verify the following relations by checking them on the generators of $\U_q(\fg)$, which are equal to the equations \eqnref{eq: K act on V}, \eqnref{eq: K act on EV}, \eqnref{eq: K act on FV},
\begin{align}
\mathcal{K}_{V}\Psi(\Delta(x))=\Delta(x)\mathcal{K}_{V}, \quad \mathcal{K}_{V}\Psi(\Delta'(x))=\Delta'(x)\mathcal{K}_V,\ \forall x\in\U_q(\fg),\label{eq: Kcomm}
\end{align}
where we have ignored $\zeta\otimes {\rm id}$ before $\Delta(x)$, $\Psi(\Delta(x))$ and etc. to simplify the notation.

\begin{proposition}\label{prop:quasiRcomm}
Retain notation above. 
Let \begin{align}\label{eq: LvLvT}
L_V:=\mathcal{K}_{V}\mathcal{R}_{V}, \quad L_V^T:=\mathcal{K}_{V}\mathcal{R}_{V}^T.
\end{align}
Then for any finite dimensional representation $(V, \zeta)$ of $\U_q(\fg)$,
\begin{align}
  L_V\Delta(x)=\Delta'(x)L_V,\quad
 L_V^T\Delta'(x)=\Delta(x)L_V^T,\quad x\in\U_q(\fg). \label{eq:R-matrix}
\end{align}
\end{proposition}
\begin{proof} It follows from \eqref{eq: quaiRcomm} that for all $x\in\U_q(\fg)$,
\[
\mathcal{R}_{V}\Delta(x)=\Psi(\Delta'(x)) \mathcal{R}_{V}.
\]
Combining this with the second relation of \eqref{eq: Kcomm}, we obtain
\[
\mathcal{K}_{V}\mathcal{R}_{V}\Delta(x)=\Delta'(x) \mathcal{K}_{V} \mathcal{R}_{V},
\]
i.e., $L_V\Delta(x)=\Delta'(x)L_V$.

The other relation of \ref{eq:R-matrix} can be similarly provced by noting that
\[
\mathfrak{R}^T\Delta'(x)=\Psi(\Delta(x))\mathfrak{R}^T, \quad \forall x\in\U_q(\fg).
\]
We remark that the universal $R$-matrix of the Drinfeld version of the quantum group in any representation of finite rank satisfies relations formally the same as \eqref{eq:R-matrix}.
\end{proof}

Now we can construct a $\Gamma_V$ satisfying the condition of Theorem \ref{thm:BGZ} as follows \cite{zhang1}.
\begin{theorem}\label{thm: Gamma}
Given any finite dimensional representation $(V,\zeta)$ of $\U_q(\fg)$, let
\begin{align}\label{eq: defgamma}
  \Gamma_{V}:=L_V^T L_V.
\end{align}
Then $\Gamma_{V}\in {\rm End}(V)\otimes \U_q(\fg)$ commutes with $(\zeta\otimes {\rm id})\Delta(x)$ for all $x\in\U_q(\fg)$
\end{theorem}
\begin{proof}
One of the relations in \propref{prop:quasiRcomm} states that $L_V\Delta(x)=\Delta'(x)L_V$ for all $x\in\U_q(\fg)$.  Thus
\[
\Gamma_V\Delta(x) =L_VL_V^T\Delta(x)=
L_V^T\Delta'(x)L_V.
\]
By using \propref{prop:quasiRcomm} again, we can re-write the right hand side as
\[
\Delta(x)L_V^TL_V=\Delta(x)\Gamma_V.
\]
Hence $\Gamma_V\Delta(x)=\Delta(x)\Gamma_V$, proving the theorem.
\end{proof}

\subsection{Central elements of the generalised quantum group}\label{sec:DtoJ}
Next we turn to the construction of central elements of the generalised quantum groups.
Let $(\zeta_{\lambda},V(\lambda))$ be a finite dimensional irreducible representation of $\U_q(\fg)$ with highest weight $\lambda\in P^+$. Define
\begin{equation}\label{eq: Celmt}
\mathcal{C}_{\lambda}^{(m)}:=\textnormal{Tr}_1((\zeta_{\lambda}(K_{2\rho})\otimes 1)\Gamma_{V(\lambda)}^m), \quad m\in\N^+.
\end{equation}
Then by \thmref{thm: Gamma} and Corollary \ref{rem:gamma-n}, $\mathcal{C}^{(m)}_{\lambda}$ belong to the centre $Z({\rm U})$ of $\U_q(\fg)$ for all $\lambda\in P^+$ and $m\in\N^+$.

\begin{remark}
We also expect that if a given $\lambda\in P^+$ satisfies certain conditions, then $\mathcal{C}^{(m)}_{\lambda}$ for finitely many $m$ generate the centre of $\U_q(\fg)$; see Section \ref{sect:rmk} for the precise statement.
\end{remark}

Let $C_{\lambda}=\mathcal{C}^{(1)}_{\lambda}$ be the central element of $\U_q(\fg)$ defined by \eqref{eq: Celmt} for $m=1$. The central elements $C_\lambda$ are particularly easy to study, since the action of $L_{V(\lambda)}$, $L_{V(\lambda)}^T$ can be written specifically.  Next we will analyze them and give the formulae of $C_{\lambda}$ briefly.

Let $$R_N:=\{(r_1,r_2,\cdots,r_N)|r_i\in\mathbb{N}\}.$$ For any $\mathbf{r}=(r_1,r_2,\cdots,r_N)\in R_{N}$, define
\begin{gather*}
F_{\mathbf{r}}:=F_{\beta_N}^{r_N}\cdots F_{\beta_2}^{r_2}F_{\beta_1}^{r_1},\ 
E_{\mathbf{r}}:=E_{\beta_N}^{r_N}\cdots E_{\beta_2}^{r_2}E_{\beta_1}^{r_1},\ 
K_{\mathbf{r}}:=K_{\beta_N}^{r_N}\cdots K_{\beta_2}^{r_2}K_{\beta_1}^{r_1},\\
D_{\mathbf{r}}:=\prod_{j=1}^Nq_{\beta_j}^{\frac{1}{2}s_j(s_j+1)} \frac{(1-q_{\beta_j}^{-2})^{s_j}}{[s_j]_{q_{\beta_j}}!},\quad
\sum \mathbf{r}=\sum_{i=1}^Nr_i\beta_i.
\end{gather*}
Then \begin{align}\label{eq: LV action}
L_{V(\lambda)}=\mathcal{K}_{V(\lambda)}\sum_{\mathbf{r}\in R}D_{\mathbf{r}}(\zeta_{\lambda}(F_{\mathbf{r}})\otimes E_{\mathbf{r}}).
\end{align}
Similarly, \begin{align}\label{eq: LV^T action}
L_{V(\lambda)}^T=\mathcal{K}_{V(\lambda)}\sum_{\mathbf{t}\in R_N}D_{\mathbf{t}}(\zeta_{\lambda}(E_{\mathbf{t}})\otimes F_{\mathbf{t}})=\sum_{\mathbf{t}\in R_N}D_{\mathbf{t}}
(\zeta_{\lambda}(E_{\mathbf{t}}K_{\mathbf{t}}^{-1})\otimes K_{\mathbf{t}}F_{\mathbf{t}})K_{V(\lambda)}.
\end{align}
where the last equation follows from \eqnref{eq: K act on EV}, \eqnref{eq: K act on FV}. 

\begin{remark}
Note that $V(\lambda)$ is finite dimensional, there are only a finite number of nonzeros in the summation terms of \eqnref{eq: LV action}, \eqnref{eq: LV^T action}.  
\end{remark}

Recall that $\Gamma_V=L_V^TL_V$. By \eqnref{eq: LV action} and \eqnref{eq: LV^T action}, we can get  
\begin{align}
\Gamma_{V(\lambda)}
=\sum_{\mathbf{r},\mathbf{t}\in R_N}D_{\mathbf{r}}D_{\mathbf{t}}                   
(\zeta_{\lambda}(E_{\mathbf{t}}K_{\mathbf{t}}^{-1})\otimes K_{\mathbf{t}}F_{\mathbf{t}})K_{V(\lambda)}\mathcal{K}_{V(\lambda)}(\zeta_{\lambda}(F_{\mathbf{r}})\otimes E_{\mathbf{r}}).\label{eq: GammaV action}
\end{align}
After substituting $\sum_{\mu\in\Pi(\lambda)}P_{\mu}^{V(\lambda)}\otimes K_{\mu}$ for $\mathcal{K}_{V(\lambda)}$ and shifting items, we can get
\begin{align}
\Gamma_{V}=\sum_{\mu\in\Pi(\lambda)}\sum_{\mathbf{r},\mathbf{t}\in R_{N}}A_{\mathbf{r},\mathbf{t}}\zeta_{\lambda}(E_{\mathbf{t}}K_{\mathbf{t}}^{-1}F_{\mathbf{r}})P_{\mu}^{V(\lambda)}\otimes F_{\mathbf{t}}K_{2\mu-2\sum\mathbf{r}+\sum\mathbf{t}}E_{\mathbf{r}},
\end{align}
where $A_{\mathbf{r},\mathbf{t}}$ arises from $D_{\mathbf{r}}$, $D_{\mathbf{t}}$, and the exchange of $F_{\mathbf{t}}$ and $K_{\mathbf{t}}$. In particular, $A_{\mathbf{r},\mathbf{t}}=1$ for $\mathbf{r}=(0,0,\cdots,0)$ and $\mathbf{t}=(0,0,\cdots,0)$,

Given that $C_{\lambda}={\rm Tr_1}((\zeta_{\lambda}(K_{2\rho})\otimes 1)\Gamma_{V(\lambda)})$, we only need to consider the items that contribute to the trace, its necessary condition is $\sum \mathbf{r}=\sum\mathbf{t}$. In addition, for any $v_{\mu}\in V(\lambda)$, with $\mu\in\Pi(\lambda)$, $\zeta_{\lambda}(K_{2\rho})v_{\mu}=q^{(2\rho,\mu)}v_{\mu}.$ Then we can get
\begin{align}\label{eq:key-formula}
C_{\lambda}=\sum_{\mu\in\Pi(\lambda)}q^{(2\rho,\mu)}\sum_{\sum\mathbf{t}=\sum\mathbf{r}}A_{\mathbf{r},\mathbf{t}}{\rm Tr}(\zeta_{\lambda}(E_{\mathbf{t}}K_{\mathbf{t}}^{-1}F_{\mathbf{r}})P_{\mu}^{V(\lambda)})F_{\mathbf{t}}K_{2\mu-\sum\mathbf{r}}E_{\mathbf{r}}.
\end{align}

So far, we have a series of central elements $C_{\lambda}$ obtained from the simple  $\U_q(\fg)$-modules $V(\lambda)$ with $\lambda\in P^+$. We will show in \thmref{mainresult}
that the subset of $C_{\lambda}$ for $\lambda$ being the fundamental weights generate the entire centre of $\U_q(\fg)$.

Recall from \eqref{eq: HCmap} the Harish-Chandra isomorphism $\gamma_{-\rho}\circ\pi$ from $Z(\U_q(\fg))$ to $({\rm U}_{ev}^0)^W$, where ${\rm U}={\rm U}_{q}({\mathfrak{g}})$. By \eqnref{eq:key-formula}, we can immediately get the following proposition.
\begin{proposition}\label{prop: gpiC}
For any $\lambda\in P^+$, we have 
\begin{align}
\label{sec3.answer}
\gamma_{-\rho}\circ\pi (C_{\lambda})=\sum_{\mu\in \Pi(\lambda)}m_{\lambda}(\mu)K_{2\mu}\in ({\rm U}_{ev}^0)^W.
\end{align}
\end{proposition}
\begin{proof}
Recall that $\pi$ is an algebra homomorphism from ${\rm U}_0$ to ${\rm U}^0$, the subalgebra of ${\rm U}$ generated by $K_{\lambda}$'s. By inspecting the formula \eqnref{eq:key-formula}, we can easily see that
$\pi(C_{\lambda})$ is the sum of terms in which $\mathbf{r}=(0,0,\cdots,0)$ and $\mathbf{t}=(0,0,\cdots,0)$. In this condition, $A_{\mathbf{r},\mathbf{t}}=1$, and ${\rm Tr}(\zeta_{\lambda}(1)P_{\mu}^{V(\lambda)})=m_{\lambda}(\mu).$ Then we can get that 
\[ \pi(C_{\lambda})=\sum_{\mu\in\Pi(\lambda)}q^{(2\rho,\mu)}m_{\lambda}(\mu)K_{2\mu}.
\]
It then follows from the definition of $\gamma_{-\rho}$ in \eqref{eq: gamma} that
\[
\gamma_{-\rho}\circ \pi(C_{\lambda})=\sum_{\mu\in \Pi(\lambda)}m_{\lambda}(\mu)K_{2\mu},
\]
where the right hand side clearly belongs to $(\U_{ev}^0)^W$.
\end{proof}

\begin{remark}
\propref{prop: gpiC} makes $C_{\lambda}$ with the complicative expression \eqref{eq:key-formula} more clear to us. Moreover, the images of them under the Harish-Chandra isomophism is similar with the character of $V(\lambda)$. This fact make it easily for us to construct the generators of the centre from the generators of $(\U_{ev}^0)^W$. 
\end{remark}

\section{Generators of the centre and their relations}
\label{sec: main-result}
\subsection{Grothendieck group of $\U_q(\fg)$}
In this section, we will write $\U=\U_q(\fg)$.

Let $K(\U)$ be the Grothendieck group over $\F$ of the category ${\rm{Rep}}_f(\U)$ of finite dimensional $\U$-modules. Since ${\rm{Rep}}_f(\U)$ is a tensor category, $K({\rm U})$ is an associative algebra, the   multiplication of which is induced by the tensor product of $\U$-modules. More explicitly, for any objects $V, V'$ in ${\rm{Rep}}_f(\U)$,  we write $[V]$ and $[V']$ for the corresponding elements in $K({\rm U})$. Then $[V][V']=[V\otimes V']$. 

We will prove $K(\U)\cong (\U_{ev}^0)^W$ in \lemref{lem: alge structure}. Thus the algebraic structure and the generators of $K(\U)$ are of crucial importance to us. 
In Lie theory, it is well known that the representation ring $R(\fg)$ for the finite dimensional simple Lie algebra $\fg$ is a polynomial algebra generated by the irreducible representations with highest weights $\varpi_1,\varpi_2,\cdots,\varpi_n$, see \cite[Theorem 23.24]{FuHarris} for details. For $K(\U)$, we have the same result.

\begin{theorem}\label{thm:generators of KU}
Let $\varpi_1,\varpi_2,\cdots,\varpi_n$ be the fundamental weights of $\fg$. Then $K(\U)$ is a polynomial algebra over $\F$ in variables $[V(\varpi_1)],[V(\varpi_2)],\cdots, [V(\varpi_n)]$, where $n$ is the rank of $\fg$.
\end{theorem}
We will take some steps to prove the \thmref{thm:generators of KU}. It is easy to verify the following lemma.
\begin{lemma}\label{lem: basis of KU}
The elements $[V(\lambda)]$ with $\lambda\in P^+$ form a basis of $K(\U)$.
\end{lemma}

Endow $P$ with the standard partial order such that $\mu \leq \lambda$ if and only if  $\lambda-\mu$ is a  non-negative integral linear combination of positive roots. Next we prove that for any $\lambda\in P^+$, $[V(\lambda)]$ is generated by $[V(\varpi_1)]$,$[V(\varpi_2)]$,$\cdots$, $[V(\varpi_n)]$, and $[V(\varpi_1)],[V(\varpi_2)],\cdots, [V(\varpi_n)]$ are algebraically independent. 

\begin{lemma}\label{lem: Cla}
There exists a  polynomial $f_{\lambda}\in \F[x_1,x_2,\dots, x_n]$ such that
\[ [V(\lambda)]=f_{\lambda}([V(\varpi_1)], [V(\varpi_2)], \dots, [V(\varpi_n)]), \quad \forall \lambda\in P^+. \]
\end{lemma}
\begin{proof}
 We use induction on $\lambda$. It is trivial when $\lambda=0$. Suppose that $\lambda=\sum_{i=1}^n k_i\varpi_i$, then the irreducible representation $V(\lambda)$ is contained in the tensor product $\bigotimes_{i=1}^n  V(\varpi_i)^{\otimes k_i}$ with multiplicity 1. This tensor product can be decomposed as
\[
\bigotimes_{i=1}^n  V(\varpi_i)^{\otimes k_i}= V(\lambda)\oplus \left(\bigoplus_{\mu \in P^+, \,\mu<\lambda} m_{\mu}V(\mu)\right),
\]
where $m_{\mu} \in \N$ denotes the multiplicity of $V(\mu)$. 
By induction, there exist $f_{\mu}\in \F[x_1\dots, x_n]$ such that $[V(\mu)]=f_{\mu}([V(\varpi_1)], [V(\varpi_2)], \dots, [V(\varpi_n)]).$ 
Therefore, we have $f_{\lambda}=x_1^{k_1}x_2^{k_2}\dots x_n^{k_n}- \sum_{\mu \in P^+, \,\mu< \lambda} m_{\mu} f_{\mu}$.
\end{proof}

Recall that any integral dominant weight $\lambda$ is a non-negative linear combination of fundamental weights, that is, $\lambda=\sum_{i=1}^n k_i\varpi_i$ for $k_i\in \N$. We can define the lexicographic order $\prec$ on  $ P^{+}=\bigoplus_{i=1}^n \N\varpi_i$.
  
\begin{lemma}\label{lem: algind}
  The elements $[V(\varpi_1)], [V(\varpi_2)], \dots, [V(\varpi_n)]$ are algebraically independent over $\F$. As a consequence, the polynomial $f_{\lambda}$ as defined in \lemref{lem: Cla} is unique.
\end{lemma}
  \begin{proof}
  Assume for contradiction  that in $\F[x_1,\dots, x_n]$ there exists
  \[ f=c_{k_1,\dots, k_n}x_1^{k_1}x_2^{k_2}\dots x_n^{k_n} + \sum_{(a_1,\dots,a_n) \prec (k_1, \dots,k_n)} c_{a_1,\dots, a_n}x_1^{a_1}x_2^{a_2}\dots x_n^{a_n} \]
  such that $f([V(\varpi_1)],[V(\varpi_2)],\cdots,[V(\varpi_n)])=0$, where all monomials in $f$ are arranged lexicographically such that $x_1^{k_1}x_2^{k_2}\dots x_n^{k_n}$ is the maximal one and $c_{k_1,\dots, k_n}\neq 0$. Note that $[V(\varpi_1)]^{k_1}[V(\varpi_2)]^{k_2}\cdots[V(\varpi_n)]^{k_n}=[\bigotimes_{i=1}^nV(\varpi_i)^{\otimes k_i}]$.
   We express this as a linear combination of the basis elements $[V(\mu)]$ with $\mu\in P^+$, then $[V(\lambda)]$ with $\lambda=\sum_{i=1}^n k_i\varpi_i$ has coefficient $1$. However,
  $[V(\lambda)]$ never appears in $[V(\varpi_1)]^{a_1}[V(\varpi_2)]^{a_2}\cdots[V(\varpi_n)]^{a_n}$ for any
  $(a_1,\dots,a_n) \prec (k_1, \dots,k_n)$.
  Therefore, $[V(\lambda)]$ appears in $f([V(\varpi_1)],[V(\varpi_2)],\cdots,[V(\varpi_n)])$ with coefficient $c_{k_1,\dots, k_n}\neq 0$, contradicting $f([V(\varpi_1)],[V(\varpi_2)],\cdots,[V(\varpi_n)])=0$.
  \end{proof}

\begin{proof}[proof of \thmref{thm:generators of KU}]
By \lemref{lem: basis of KU}, \lemref{lem: Cla}, \lemref{lem: algind}, we can immediately complete the proof.
\end{proof}

\subsection{Explict generators of the centre}
The following theorem is the main result of this paper. 

Recall that $C_{\lambda}=\mathcal{C}_{\lambda}^{1}$ is defined in \eqnref{eq: Celmt}, which has the explicit expression in \eqnref{eq:key-formula}, and $\varpi_1,\varpi_2,\cdots,\varpi_n$ are the fundamental weights of $\mathfrak{g}$.
\begin{theorem}\label{mainresult}
The centre $Z({\rm U})$ of ${\rm U}_q(\fg)$ is the polynomial algebra over $\F$ in the variables $C_{\varpi_1}, C_{\varpi_2},\dots, C_{\varpi_n}$, i.e., $Z({\rm U})\cong\F[C_{\varpi_1}, C_{\varpi_2},\dots, C_{\varpi_n}]$.
\end{theorem}

\begin{proof}
By the quantised Harish-Chandra isomorphism in Theorem \ref{thm:HCiso}, $Z({\rm U})$ is isomorphic to $(\U_{ev}^0)^W$ as algebras. Thus \thmref{mainresult} can be proven by showing that $(\U_{ev}^0)^W$ is a polynomial algebra over $\F$ in the variables  $\gamma_{-\rho}\circ \pi(C_{\varpi_i})$ for $1\leq i\leq n$. This is proven in Corollary \ref{cor:poly} below.
\end{proof}

The remainder of this section is devoted to the proof of Corollary \ref{cor:poly}. This will be done by establishing a series of lemmas.
Now the following result is clear.
\begin{lemma}\label{lem:basis-1}
Set for any $\lambda\in P$, $${\rm av}(\lambda)=\sum_{\mu\in W\lambda}K_{2\mu},$$ then ${\rm av}(\lambda)$ for all $\lambda\in P^+$ form a basis of $(\U_{ev}^0)^W.$
\end{lemma}

\begin{lemma}\label{lem: alge structure}
There is an algebra isomorphism 
\begin{eqnarray}\label{eq:Ch}
{\rm Ch}: K(\U)\longrightarrow (\U_{ev}^0)^W,\quad {\rm Ch}([V])=\sum_{\mu\in\Pi(V)}{\rm dim}V_{\mu}K_{2\mu},
\end{eqnarray}
where $\Pi(V)$ is the set of weights of $V$ and ${\rm dim}V_{\mu}$ is the dimension of the weight space $V_{\mu}$.
\end{lemma}

\begin{proof}
It is easy to verify that $\rm Ch$ is an algebra homomorphism.
Firstly, we prove that it is injective.
For any $[V],[W]\in K(\U)$ such that ${\rm Ch}([V])={\rm Ch}([W])$, by the definition of $\rm Ch$, we have $$\sum_{\mu_1\in\Pi(V)}{\rm dim}V_{\mu_1}K_{2\mu_1}=\sum_{\mu_2\in\Pi(W)}{\rm dim}W_{\mu_2}K_{2\mu_2}.$$
Since $K_{\mu}$ with $\mu\in P$ are linear independent, then we can get that $\Pi(V)=\Pi(W)$, and for any $\mu\in\Pi(V)$, ${\rm dim}V_{\mu}={\rm dim}W_{\mu}$, thus $[W]=[V]$.

Next we prove it is surjective. By \lemref{lem:basis-1}, ${\rm av}(\lambda)$ for all $\lambda\in P^+$ form a basis of $(\U_{ev}^0)^W$. We need to show that all these ${\rm av}(\lambda)$ are in the image of $\rm Ch$. Note that $P$ is endowed with the standard partial order that $\mu \leq \lambda$ if and only if  $\lambda-\mu$ is a  non-negative integral linear combination of positive roots.
We use upward induction on the partial ordering of $P^+$. Starting with $\lambda$ minimal, i.e., no other $\mu\in P^+$ can occur as a weight of $V(\lambda)$, then we have ${\rm Ch}([V(\lambda)])={\rm av}(\lambda)$.
Suppose that $\lambda=\sum_{i=1}^n k_i\varpi_i$, then the irreducible representation $V(\lambda)$ is contained in the tensor product $\bigotimes_{i=1}^n  V(\varpi_i)^{\otimes k_i}$ with multiplicity 1. This tensor product can decomposes as the direct sum of $V(\mu)$ with $\mu\leq\lambda$.
Recall that $\Pi(\mu)$ is the set of weights of $V(\mu)$, which is $W$-invariant, and ${\rm dim}V_{\nu}={\rm dim}V_{w\nu}$ for any $\nu\in\Pi(\mu)$, $w\in W$. Then we can get 
\begin{align*}
\bigotimes_{i=1}^n  V(\varpi_i)^{\otimes k_i}=\left(\bigoplus_{w\in W}V_{w\lambda}\right)\oplus\left(\bigoplus_{w\in W,\eta\in P^+,\eta<\lambda}{\rm dim}V_{\eta}V_{w\eta}\right).
\end{align*}
Note that ${\rm Ch}(V(\varpi_i))={\rm av}(\varpi_{i})$.
Applying the ring homomorphism ${\rm Ch}$ to both sides, we obtain
 \[ \prod_{i=1}^n{\rm av}(\varpi_i)^{k_i}={\rm av}(\lambda)+\sum_{\eta\in P^+,\eta<\lambda}{\rm dim}V_{\eta}{\rm av}(\eta).\]
By induction, there exist inverse images for all ${\rm av}(\eta)$ with $\eta<\lambda$, then so is ${\rm av}(\lambda)$.
\end{proof}

Combining the \propref{prop: gpiC}, we can get the following corollary.
\begin{corollary}\label{cor:poly}
Let $\widetilde{C}_{\lambda}=\gamma_{-\rho}\circ\pi(C_{\lambda})$, then 
$(\U_{ev}^0)^W\cong \F[\widetilde{C}_{\varpi_1}, \widetilde{C}_{\varpi_2}, \dots, \widetilde{C}_{\varpi_n}]$, the polynomial algebra in the $n$ variables $\widetilde{C}_{\varpi_1},\widetilde{C}_{\varpi_2},,\cdots,\widetilde{C}_{\varpi_n}$. 
\end{corollary}
\begin{proof}
By \propref{prop: gpiC} and the definition of ${\rm Ch}$, we have ${\rm Ch}([V(\lambda)])=\widetilde{C}_{\lambda}.$ In pariticular, ${\rm Ch}([V(\varpi_i)])=\widetilde{C}_{\varpi_i}$. By \thmref{thm:generators of KU} and \lemref{lem: alge structure}, we can immediately complete the proof.
\end{proof}

\subsection{Some remarks}\label{sect:rmk}
We have shown that the subset of central elements $\mathcal{C}^{(m)}_{\lambda}$ defined by \eqref{eq: Celmt}, with $m=1$ and $\lambda$ being the fundamental weights,   generates the centre of $\U_q(\fg)$.
We also expect the following to be true.
\begin{conjecture}\label{conj}
If tensor powers of  a finite dimensional simple $\U_q(\fg)$-module $V(\lambda)$ separate points of $\U_q(\fg)$ (see Proposition \ref{sec1.injection}), then there is a finite subset $\mathfrak{M}_{\fg, \lambda}$ of $\N^+$ such that $\{\CC^{(m)}_\lambda\mid m\in\mathfrak{M}_{\fg, \lambda}\}$ generates the centre of $\U_q(\fg)$.
\end{conjecture}
This is the case  \cite{Junbo} for $\fg=\mathfrak{gl}_n$ and $V(\lambda)=\F^n$ being the natural module; and one can also easily extract from op. cit.  such a set of generators for the center of $\U_q(\fsl_n)$.
It will be very interesting to prove this for $\U_q(\fg)$ for the other simple Lie algebras $\fg$ by identifying such $\lambda$ and the corresponding minimal sets $\mathfrak{M}_{\fg, \lambda}$.

\begin{appendix}
\section{Proof of the Harish-Chandra isomorphism}\label{app: proof}
This part is about the algebraic proof of \thmref{thm:HCiso}, i.e., the quantised Harish-Chandra isomophism of $\U_q(\fg)$.
Note that it can be proven in much the same way as the proof in \cite[Chapter 6]{Jantzen}. However, we can hardly find a proof in detail with the method developed in \cite{Jantzen}.
Hence, we give some pertinent steps in the following.

Write $\U=\U_q(\fg)$. We first show that $\gamma_{-\rho}\circ\pi$ indeed maps  $Z({\rm U})$ into the invariant subalgebra $({\rm U}_{ev}^0)^W$.
  
Observe the following elementary result.
\begin{lemma}\label{sec2.inj}
Let $\lambda\in P$. Any $u\in Z({\rm U})$ acts on the Verma module $M(\lambda)$ as a scalar multiplication by $\chi_{\lambda}(\pi(u))$.
\end{lemma}

As an immediate consequence, we have
  
\begin{lemma}\label{lem: HCinj}
The restriction of $\pi$ to $Z({\rm U})$ is injective, and hence so
is $\gamma_{-\rho}\circ\pi$.
\end{lemma}
\begin{proof}
If $\pi(u)=0$, then by Lemma \ref{sec2.inj}, we have $u.M(\lambda)=0$ and  hence $u.V(\lambda)=0$ for all $\lambda\in P^+$. By Proposition \ref{sec1.injection}, $u=0$.
\end{proof}
  
We now show that the image $\gamma_{-\rho}\circ\pi(Z({\rm U}))$ of the centre is invariant under the Weyl group action.
  
\begin{lemma}
\label{sec2.inva}
The images of $Z({\rm U})$ under the Harish-Chandra isomophism are all in $({\rm U}^0)^W$, i.e.,  $\gamma_{-\rho}\circ\pi(Z({\rm U}))\subseteq ({\rm U}^0)^W$.
\end{lemma}
\begin{proof}
Fix any central element $u\in Z(\rm U)$, we write  $h=\gamma_{-\rho}\circ\pi(u)$.
  
Given any $\lambda\in P$ and $i\in \{1,2,\cdots,n\}$, we let $\mu=s_{\alpha_i}(\lambda+\rho)-\rho$.
  
If $(\lambda,\alpha_i^{\vee})\geq 0$, there is a nontrivial homomorphism $M(\mu)\rightarrow M(\lambda)$ \cite[Chapter 5.9]{Jantzen}.
By Lemma \ref{sec2.inj},
\begin{align}
\label{eq: wh=h}
\chi_{\lambda+\rho}(h)=\chi_{\mu+\rho}(h)=\chi_{\lambda+\rho}(s_{\alpha_i}h).
\end{align}
  
If $(\lambda,\alpha_i^{\vee})<-1$, then $(\mu,\alpha_i^{\vee})$ is non-negative, thus we may apply the above arguments to $\mu$ to show that \eqref{eq: wh=h} still holds.
  
Then the only other possibility is that  $(\lambda,\alpha_i^{\vee})=-1$. In this case $\mu=\lambda$, and \eqref{eq: wh=h} holds trivially.
  
Since \eqref{eq: wh=h} holds for all $\lambda$ and $i$, and $s_{\alpha_i}$ generate $W$, we have
\begin{align}
\label{eq: wh-h=0}
\chi_{\lambda}(wh-h)=0,\quad \forall w\in W, \ \lambda\in P.
\end{align}
We can always write $wh-h=\sum_{\eta}a_{\eta}K_{\eta}$. Then \eqref{eq: wh-h=0} leads to
\[
\sum\limits_{\eta}a_{\eta}\chi_{\lambda}(K_\eta)=\sum\limits_{\eta}a_{\eta}q^{(\lambda, \eta)}= \sum\limits_{\eta}a_{\eta}\chi_{\eta}(K_\lambda)=0, \quad \forall \lambda\in P.
\]
Thus $\sum\limits_{\eta}a_{\eta}\chi_{\eta}=0$.
The linear independence of characters then implies $a_{\eta}=0$ for all $\eta$. Hence $wh-h=0$ for all $w\in W$, i.e.,  $h\in ({\rm U}^0)^W$ as claimed.
\end{proof}
  
Now the following lemma justifies the range of $\gamma_{-\rho}\circ\pi$ as  defined in \eqref{eq: HCmap}.
  
\begin{lemma}
\label{sec2.invainva}
The Harish-Chandra homomorphism $\gamma_{-\rho}\circ\pi$ maps $Z({\rm U})$ to $({\rm U}_{ev}^0)^W$.
\end{lemma}
\begin{proof}
Take an arbitrary $u\in Z({\rm U})$, and write
\[
\gamma_{-\rho}\circ\pi(u)=\sum\limits_{\mu\in P}a_{\mu}K_{\mu}.
\]
By Lemma \ref{sec2.inva}, $\gamma_{-\rho}\circ \pi(u)\in ({\rm U}^0)^W$. Thus $a_{w\mu}=a_{\mu}$ for all $w\in W$ and $\mu\in P$. We have to show that $a_{\mu}\neq0$ only if $\mu\in2 P.$
  
Recall from (\ref{sec2.widehatsigma}) that there is an automorphism $\psi_{\sigma}$ of ${\rm U}$ associated to each group character $\sigma$ as defined in \eqref{eq: gpchar}. It can be easily verified that $\psi_{\sigma}$ commutes with both $\pi$ and $\gamma_{-\rho}$. Therefore, we have
$$\gamma_{-\rho}\circ\pi(\psi_{\sigma}(u))=\psi_{\sigma}(\sum\limits_{\mu}a_{\mu}K_{\mu})=\sum\limits_{\mu}a_{\mu}\sigma(\mu) K_{\mu},$$
which lands in $({\rm U}^0)^W$ since  $\psi_{\sigma}(u)$ is central. It follows that
$$a_{\mu}\sigma(\mu)=a_{w\mu}\sigma(w\mu)=a_{\mu}\sigma(w\mu)\quad  \forall w\in W, \mu\in P.$$
Since we have assumed that $a_{\mu}\neq 0$, this in particular implies $1=\sigma(\mu-s_{\alpha_i}\mu)$ for $1\leq i\leq n$. Fixing a group character $\sigma: P\to \mathbb{C}^{\times }$ such that $\sigma(\alpha_i)=-1$ for all $i$, we have
$$\sigma(\mu-s_{\alpha_i}\mu)=\sigma((\mu, \alpha_i^{\vee})\alpha_i)=(-1)^{(\mu, \alpha_i^{\vee})}=1.$$
This implies that $(\mu, \alpha_i^{\vee})$ is even for  $1\leq i\leq n$, i.e.,  $\mu\in 2P$.
\end{proof}

Now we prove the quantum Harish-Chandra isomorphism following \cite[Chapter 6]{Jantzen}.
  
\subsection{Proof of the isomorphism}
By \lemref{lem: HCinj},  the restriction of $\gamma_{-\rho}\circ\pi$ to $Z({\rm U})$ is injective. Therefore, it suffices to show surjectivity of the map \eqref{eq: HCmap} in order to prove \thmref{thm:HCiso}. We do this by showing that each basis element of the invariant subalgebra $({\rm U}_{ev}^0)^W$ has a pre-image in $Z({\rm U})$.

We will follow the strategy of \cite{Jantzen} to prove the surjectivity. This relies in an essential way on a non-degenerate bilinear form on ${\rm U}$,  which can be constructed in exactly the same way as in
\cite[Chapter 6]{Jantzen}.
However, the explicit construction is rather involved and technical. We will merely describe the main properties of the form here, and refer to op. cit. for details.
  
\begin{lemma}\cite[Chapter 6]{Jantzen}\label{lem: Upairing}
There exists a unique bilinear form
\[ (\ ,\ ): {\rm U}^{\leq 0} \times {\rm U}^{\geq 0} \rightarrow \F \]
with the following properties:
\[
\begin{aligned}
      &(K_{\lambda}, K_{\mu})=q^{-(\lambda, \mu)}, &\quad& (K_{\lambda}, E_i)=0,\\
    & (F_i, E_j)= \delta_{ij} (q_i-q_i^{-1})^{-1}, &\quad & (F_i, K_{\lambda})=0,\\
    &(x, y_1y_2)=(\Delta(x), y_1\otimes  y_2), &\quad& (x_1x_2,y)=(x_1\otimes x_2, \Delta(y)),
\end{aligned}
\]
for all $x, x_1, x_2\in {\rm U}^{\leq 0}$, $y,y_1,y_2\in  {\rm U}^{\geq 0}$, $\lambda, \mu \in P$ and $1\leq i,j\leq n$.
\end{lemma}
  
\begin{proposition}\cite[Chapter 6]{Jantzen}
\label{prop: Upairing}
Let $\lambda,\eta\in P$, $\mu,\nu\in Q^+.$
\begin{enumerate}
 \item  $(xK_{\lambda},yK_{\eta})=q^{-(\lambda,\eta)}(x,y)$ for any $x\in {\rm U}^-$ and $y\in {\rm U}^+$.
  \item $({\rm U}^{-}_{-\nu} ,{\rm U}_{\mu}^+)=0$  for any $\mu\neq\nu$.
  \item The restriction $(\ ,\ )|_{{\rm U}_{-\mu}^-\times{\rm U}_{\mu}^+}$ is non-degenerate.
\end{enumerate}
\end{proposition}

We now define a bilinear form on ${\rm U}$ by using \lemref{lem: Upairing}. Recall that ${\rm U}^+$ (resp. ${\rm U}^-$) is $Q^+$-graded (resp. $Q^-$-graded) vector space with respect to the ${\rm U}^0$-action given in \eqref{eq: grading}, and the multiplication induces an isomorphism ${\rm U}^-\otimes {\rm U}^0 \otimes {\rm U}^+\cong {\rm U}$.  Since $K_{\mu}$ is a unit in  ${\rm U}$, we can rearrange this isomorphism into
\[ \bigoplus_{\mu,\nu \in Q^+}  {\rm U}^-_{-\mu}K_{\mu} \otimes {\rm U}^0 \otimes {\rm U}^+_{\nu} \cong {\rm U}. \]
Now  the bilinear form
$\langle\ ,\ \rangle: {\rm U}\times {\rm U}\rightarrow \F$
is defined on the graded components  by
\begin{equation}\label{eq: BilinearForm}
\langle yK_{\nu}K_{\lambda}x,y'K_{\nu'}K_{\eta}x'\rangle:=(y',x)(y,x')q^{(2\rho,\nu)}(q^{1/2})^{-(\lambda,\eta)}
\end{equation}
for all $x\in {\rm U}_{\mu}^+, x'\in {\rm U}_{\mu'}^+$, $y\in {\rm U}_{-\nu}^-$, and $y'\in {\rm U}_{-\nu'}^-$, with $\lambda,\eta\in P,\mu,\mu',\nu,\nu'\in Q^+$. It follows immediately from part (2) of \propref{prop: Upairing} that
\[ \langle {\rm U}_{-\nu}^-{\rm U}^0{\rm U}_{-\mu}^+,{\rm U}_{-\nu'}^-{\rm U}^0{\rm U}_{\mu'}^+\rangle=0,\quad \text{unless}\ \mu=\nu', \nu=\mu'.\]
  
The following proposition gives two significant properties for the bilinear form \eqref{eq: BilinearForm}, which will be used in the proof of surjectivity of the Harish-Chandra homomorphism.
\begin{proposition}\cite[Chapter 6]{Jantzen}
\label{sec3.prop of second bilinear form}
\begin{enumerate}
\item If $\langle v,u\rangle=0$ for all $v\in{\rm U}$, then $u=0$;
\item $\langle\textnormal{ad}(x)u,v\rangle=\langle u,\textnormal{ad}(S(x))v\rangle$ for all $x,u,v\in{\rm U}$.
\end{enumerate}
\end{proposition}

Let $M$ be a finite dimensional ${\rm U}$-module. For any $m\in M$ and $f\in M^*$, let $c_{f,m}\in {\rm U}^*$ be the linear form with $c_{f,m}(v)=f(vm)$ for any $v\in{\rm U}.$ The following lemma follows from the non-degeneracy of the form $\langle\ , \ \rangle$ \cite[Chapter 6.22]{Jantzen}.
\begin{lemma}\label{lem: cfm}
Retain notation above. There
exists a unique element $u\in {\rm U}$, depending on $f\in M^*$, $m\in M$ such that
\[ c_{f,m}(v)=\langle v,u\rangle, \quad \forall v\in {\rm U}.   \]
\end{lemma}
  
This leads to the following key lemma.
  
\begin{lemma}\label{lem: keylem}
Fix $\lambda\in P^+$, and let $V(\lambda)$ be the finite dimensional simple ${\rm U}$-module with highest weight $\lambda$. Then there exists a unique central element $z_{\lambda}\in Z({\rm U})$ such that
\begin{equation}\label{eq: zdef}
\langle u,z_{\lambda}\rangle={\rm Tr}(uK_{2\rho}^{-1}), \quad \forall u\in {\rm U},
\end{equation}
where ${\rm Tr}(x)$ denotes the trace of $x\in {\rm U}$ over $V(\lambda)$.
\end{lemma}
\begin{proof}
Let $m_1,m_2,\cdots,m_{r}$ be a basis of $V(\lambda)$ and $f_1,f_2,\cdots,f_r$ the dual basis of $V(\lambda)^*$, i.e., $f_i(m_j)=\delta_{ij}$. Then the trace of $uK_{2\rho}^{-1}$ over $V(\lambda)$ is equal to
$\sum\limits_{i=1}^r c_{f_i,K_{2\rho}^{-1}m_i}(u).$
By \lemref{lem: cfm}, there is a unique $v_i\in \U$ such that $\langle u,v_i\rangle=c_{f_i,K_{2\rho}^{-1}m_i}(u)$ for all $u\in\U$.  Let $z_{\lambda}=v_1+v_2+\cdots+v_r$, then we have $\langle u,z_{\lambda}\rangle=\sum\limits_{i=1}^r c_{f_i,K_{2\rho}^{-1}m_i}(u)$, which is the trace of $uK_{2\rho}^{-1}$ over $V(\lambda)$.
  
It remains to show that $z_{\lambda}$ is central in ${\rm U}$, which is equivalent to showing that $\textnormal{ad}(u)z_{\lambda}=\varepsilon(u)z_{\lambda}$ for any $ u\in {\rm U}$.  Then the linear representation $\varsigma_{\lambda}:{\rm U}\rightarrow \End(V(\lambda))$  is a homomorphism of ${\rm U}$-modules, where ${\rm U}$ acts on itself by the adjoint action, that is,  $u.v:={\rm ad}(u)v$ for any $u,v\in {\rm U}$.
The quantum trace $\textnormal{Tr}_{q}:\End(V(\lambda))\rightarrow \F$ that takes $\varphi\mapsto \textnormal{Tr}(\varphi\circ K_{2\rho}^{-1})$ is also a ${\rm U}$-module homomorphism, where $\F$ is the trivial module such that $u.a=\varepsilon(u)a$ for any $a\in \F$. Let $\theta={\rm Tr}_q\circ  \varsigma_{\lambda}$. Then by definition
\[ \theta(u)= {\rm Tr}_q\circ  \varsigma_{\lambda}(u)={\rm Tr}(uK_{2\rho}^{-1})=\langle u,z_{\lambda}\rangle, \quad \forall u\in  {\rm U}. \]
Since $\theta$ is a  ${\rm U}$-module homomorphism, we have
\[ \theta(u.v)=u.\theta(v)= \varepsilon(u)\theta(v)=\varepsilon(u)\langle v,z_{\lambda}\rangle. \]
On the other hand, using the adjoint structure of ${\rm U}$ we have
\[ \theta(u.v)={\rm Tr}_q\circ  \varsigma_{\lambda}({\rm ad}(u)v)=\langle \textnormal{ad}(u)v,z_{\lambda}\rangle=\langle v,\textnormal{ad}(S(u))z_{\lambda}\rangle.   \]
where the last equation follows from part(2) of Proposition \ref{sec3.prop of second bilinear form}. Since the bilinear form is non-degenerate we have  $\textnormal{ad}(S(u))z_{\lambda}=\varepsilon(u)z_{\lambda}$ for all $u\in {\rm U}$.  Recalling that  the antipode $S$ satisfies $\varepsilon\circ S=\varepsilon$, we obtain $\textnormal{ad}(u)z_{\lambda}=\varepsilon(u)z_{\lambda}$ for all $u\in {\rm U}$. Therefore, $z_{\lambda}\in Z({\rm U})$.
\end{proof}
\begin{lemma}
\label{sec3.gamma -rho zlambda}
Let $\lambda\in P^+$, and  $V(\lambda)$ the finite dimensional simple module of ${\rm U}$. Let $z_{\lambda}\in Z({\rm U})$ be  the central element  defined in \eqref{eq: zdef}. Then
$$\gamma_{-\rho}\circ\pi(z_{\lambda})=\sum\limits_{\eta\in\Pi(\lambda)}m_{\lambda}(\eta)K_{-2\eta},$$
where $\Pi(\lambda)$ is the set of weights of $V(\lambda)$ and $m_{\lambda}(\eta)$ denotes the dimension of  the weight space $V(\lambda)_{\eta}$.
\end{lemma}
\begin{proof}
Since $z_{\lambda}$ is central and $Z({\rm U}) \subseteq \U_0={\rm U}^0\oplus\bigoplus\limits_{\nu>0}{\rm U}_{-\nu}^-{\rm U}^0{\rm U}_{\nu}^+$, we may write
\[
z_{\lambda}=z_{\lambda,0} + \sum_{\nu> 0}z_{\lambda,\nu},\quad \text{with }  z_{\lambda,0}\in\U^0, \, z_{\lambda,\nu}\in \U_{-\nu}^-{\rm U}^0{\rm U}_{\nu}^+.
\]
It follows that $\pi(z_{\lambda})= z_{\lambda,0}$.
By \eqref{eq: BilinearForm}, we have
\begin{equation}\label{eq: Kz1}
 \langle K_{\mu},z_{\lambda}\rangle=\langle K_{\mu},z_{\lambda,0}\rangle= \langle K_{\mu}, \pi(z_{\lambda})\rangle, \quad \forall \mu\in P.
\end{equation}
On the other hand, using \lemref{lem: keylem} we obtain
\begin{equation}\label{eq: Kz2}
\begin{aligned}
\langle K_{\mu},z_{\lambda}\rangle&={\rm Tr}(K_{\mu-2\rho})= \sum\limits_{\eta\in\Pi(\lambda)}m_{\lambda}(\eta)q^{(\eta,\mu-2\rho)}\\
 &=\sum\limits_{\eta\in\Pi(\lambda)}m_{\lambda}(\eta)q^{-(2\eta,\rho)}q^{(\mu,\eta)}\\
&= \sum\limits_{\eta\in\Pi(\lambda)}m_{\lambda}(\eta)q^{-(2\eta,\rho)}\langle K_{\mu}, K_{-2\eta}\rangle.
\end{aligned}
\end{equation}
Comparing \eqref{eq: Kz1} and \eqref{eq: Kz2} and using the non-degeneracy of the bilinear form, we have
\[
\gamma_{-\rho}\circ\pi(z_{\lambda})=
\sum\limits_{\eta\in\Pi(\lambda)}m_{\lambda}(\eta)K_{-2\eta}.
\]
This completes the proof.
\end{proof}

Now we are ready to prove \thmref{thm:HCiso}.

\begin{proof}[Proof of Theorem~\ref{thm:HCiso}]
We know that   $\gamma_{-\rho}\circ\pi$ is injective from \lemref{lem: HCinj}.  It remains to show that $\gamma_{-\rho}\circ\pi$ is surjective.  By \lemref{lem:basis-1}, the elements ${\rm av}(-\mu)= \sum_{\eta\in W\mu}K_{-2\eta}$ with $\mu\in P^+$ form a basis for $({\rm U}_{ev}^0)^W$, since each group orbit $W\mu$ in $ P$ contains exactly one $-\mu$ such that $\mu$ is dominant.
  
We use induction on $\mu$ to show that the basis elements ${\rm av}(-\mu)$ are  in the image of $\gamma_{-\rho}\circ\pi$. Endow $P$ with the standard partial order such that $\mu \leq \lambda$ if and only if  $\lambda-\mu$ is a  non-negative integral linear combination of positive roots. For the base case $\nu=0$, we have $\text{av}(0)=1=\gamma_{-\rho}\circ\pi(1)$.  For any $\lambda\in P^+$, we may  apply  Lemma \ref{lem: keylem}  and then obtain the  element $z_{\lambda}\in Z({\rm U})$, which by  Lemma \ref{sec3.gamma -rho zlambda} has the image
$$\gamma_{-\rho}\circ\pi(z_{\lambda})=\sum\limits_{\eta\in\Pi(\lambda)}m_{\lambda}(\eta)K_{-2\eta}=\textnormal{av}(-\lambda)+\sum\limits_{\mu<\lambda,\mu \in P^+}m_{\lambda}(\mu)\textnormal{av}(-\mu),$$
where the second equality follows from the fact that  $m_{\lambda}(\lambda)=m_{\lambda}(w\lambda)=1$ for any  $w\in W$.  The left hand side of the above equation belongs to
$\gamma_{-\rho}\circ\pi(Z({\rm U}))$. By induction  hypothesis, all $\text{av}(-\mu)$ with $\mu<\lambda$  are in the image of $\gamma_{-\rho}\circ\pi$, hence so is also $\textnormal{av}(-\lambda).$
\end{proof}
\end{appendix}

\end{document}